\documentclass[12pt,reqno]{amsart}
\usepackage[latin1]{inputenc}
\usepackage{mathptmx}
\usepackage{pstricks}

\def\vertex{\pscircle[fillstyle=solid,fillcolor=black]{0.05}}

\newtheorem{lemma}{Lemma}[section]
\newtheorem{corollary}[lemma]{Corollary}
\newtheorem{theorem}[lemma]{Theorem}
\newtheorem{proposition}[lemma]{Proposition}

\theoremstyle{definition}

\newtheorem{remark}[lemma]{Remark}

\newtheorem{question}[lemma]{Question}

\def\ini{\operatorname{in}}
\def\Hilb{\operatorname{Hilb}}
\def\Hom{\operatorname{Hom}}
\def\Proj{\operatorname{Proj}}
\def\Spec{\operatorname{Spec}}
\def\relint{\operatorname{int}}
\def\conv{\operatorname{conv}}
\def\vert{\operatorname{vert}}
\def\aff{\operatorname{aff}}
\def\rank{\operatorname{rank}}
\def\Pic{\operatorname{Pic}}

\def\RR{{\mathbb R}}

\def\ZZ{{\mathbb Z}}
\def\PP{{\mathbb P}}

\def\cF{{\mathcal F}}
\def\cN{{\mathcal N}}
\def\cV{{\mathcal V}}
\def\cW{{\mathcal W}}
\def\cC{{\mathcal C}}
\def\cS{{\mathcal S}}

\def\mm{{\mathfrak m}}

\title{The quest for counterexamples in toric geometry}

\author{Winfried Bruns}
\address{Universität Osnabrück, Institut für Mathematik, 49069 Osnabrück, Germany}
\email{wbruns@uos.de}

\begin{document}

\maketitle

\begin{abstract}
We discuss an experimental approach to open problems in toric
geometry: are smooth projective toric varieties (i)
arithmetically normal and (ii) defined by degree $2$ equations?
We discuss the creation of lattice polytopes defining smooth
toric varieties as well as algorithms checking properties (i)
and (ii) and further potential properties, in Particular a
weaker version of (ii) asking for scheme-theoretic definition
in degree $2$.
\end{abstract}

\section{Introduction}

Two of the most tantalizing questions in toric geometry concern
the arithmetic normality and the degree of the defining
equations of smooth projective toric varieties:
\begin{enumerate}
\item[(N)] (Oda) Is every equivariant embedding of such a
    variety $\cV$ into projective space arithmetically
    normal? Inn other words, is the homogeneous coordinate
    ring normal?
\item[(Q)] (B\o gvad) Is the ideal of functions vanishing
    on $\cV$ generated in degree $2$?
\end{enumerate}
Both questions have affirmative answers in dimension $2$, but
are open in dimension $\ge 3$. They were the major themes of
workshops at the Mathematisches Institut Oberwolfach (2007) and
the American Institute of Mathematics (2009). The Oberwolfach
report \cite{OW} gives a good overview of the subject. See
Ogata \cite{OgaProj}, \cite{OgaFiber} for some positive results
in dimension $3$.

Toric geometry has developed a rather complete dictionary that
translates properties of projective toric varieties into
combinatorics of lattice polytopes, and therefore both
questions can be formulated equivalently in the language of
lattice polytopes and their monoid algebras. In the following,
lattice polytopes representing smooth projective toric
varieties are called \emph{smooth}, those representing a normal
projective toric variety are called \emph{very ample}, and
those representing an arithmetically normal subvariety are
\emph{normal}. A very brief overview of the connection between
toric varieties and lattice polytopes is given in Section
\ref{Toric}.

An algorithmic approach for the search of counterexamples was
discussed by Gubeladze and the author about 10 years ago, and
taken up by Gubeladze and Ho\c{s}ten in 2003, however not fully
implemented. Such an implementation was realized by the author
in 2007, and completed and augmented in several steps. A
software library on which the implementation is based had
previously been developed for the investigation of unimodular
covering and the integral Carathéodory property \cite{BG:cov},
\cite{ICP}. Moreover, Normaliz \cite{Nmz} proved very useful
(and profited from the experience gained in this project).

Unfortunately the search for counterexamples has been fruitless
to this day. Nevertheless we hope that a discussion of the
algorithmic approach to (Q) and (N) and several related
properties of smooth lattice polytopes is welcome.

The main experimental line consists of three computer programs
for the following tasks:
\begin{enumerate}
\item the random creation of smooth projective toric
    varieties via their defining fans;
\item the computation of support polytopes;
\item the verification of various properties, in particular
    (N) and (Q).
\end{enumerate}
The implementation of the first two tasks is described in
Section \ref{Fans}.

Testing normality (Section \ref{Norm}) is much easier and
faster than testing quadratic generation, and amounts to a
Hilbert basis computation that is usually a light snack for the
Normaliz algorithm described in \cite{BK} and \cite{BI}.
Quadratic generation requires more discussion (Section
\ref{Quadr}). One of the results found in connection with the
experimental approach is a combinatorial criterion for
scheme-theoretic definition in degree $2$. In contrast to
ideal-theoretic definition in degree $2$, as asked for in (Q),
it can be tested efficiently for polytopes with a large number
of lattice points.

For an arbitrary lattice polytope $P$ the multiples $cP$ are
normal for $c\ge \dim P-1$ and their toric ideals are generated
in degree $2$ for $c\ge \dim P$ \cite{BGT}. Therefore one
expects counterexamples to have few lattice points, and so we
try to reduce smooth lattice polytopes in size without giving
up smoothness, of course. In Section \ref{Chisel} we explain
the technique of \emph{chiseling}, already suggested by
Gubeladze and Ho\c{s}ten, that splits a smooth polytope in two
parts unless it is \emph{robust}. It is then not hard to see
that a minimal counterexample to (Q) or (N) must be robust.

After a discussion of some further potential properties of
smooth polytopes, in particular the positivity of the
coefficients of their Ehrhart polynomials, we widen the class
of lattice polytopes by including the very ample ones. In fact,
it would already be very interesting to find simple polytopes
that are very ample but not normal.

One can interpret the failure of the search for counterexamples
as an indication that (N) and (Q) hold. However, the main
difficulty is not the investigation of given smooth polytopes:
it is their construction for which we depend on the
construction of fans, objects that live in the space dual to
that of the polytopes. It is doubtful whether we can generate a
sufficient amount of complexity in the dual space without
loosing the passage to primal space. An argument supporting
this viewpoint is given in Section \ref{VeryA}.

The software on which our experiments have been based was made
public in 2009 and has recently been updated \cite{TE}. Its
documentation discusses the practical aspects of its use. They
will be skipped in the following.
\bigskip

\emph{Acknowledgement.}\enspace The author is very grateful to
Joseph Gubeladze and Serkan Ho\c{s}ten for sharing their ideas.
Without Joseph's enthusiasm, his inspiring comments and the
perpetual discussions with him (almost never controversial),
this project would not have been started.

The author is also indebted to Mateusz Michalek for his careful
reading of the paper, in particular for pointing out a mistake
in the author's first version of Theorem \ref{scheme2} and for
suggesting the correction.

\section{Lattice polytopes and toric varieties}\label{Toric}

We assume that the reader is familiar with the basic notions of
discrete convex geometry, combinatorial commutative algebra and
toric algebraic geometry. These notions are developed in the
books by Fulton \cite{Ful}, Oda \cite{Oda} and
Cox--Little--Schenck \cite{CLS}. We will follow the terminology
and notation of \cite{BG}.

Nevertheless, let us briefly sketch the connection between
projective toric varieties, projective fans and lattice
polytopes since the experimental approach to the questions (N)
and (Q) is based on it. The main actors in the experimental
approach are lattice polytopes, and therefore it is natural,
but opposite to the conventions of toric geometry, to have them
live in the primal vector space $V=\RR^d$ whereas (their
normal) fans reside in the dual space $V^*=\Hom_\RR(V,\RR)$. By
$L$ we denote the lattice $\ZZ^d$, and $L^*$ is its dual in
$V^*$.

Let $P\subset V$ be a lattice polytope. In order to avoid
technicalities of secondary importance we will assume that the
lattice points in $P$ generate $L$ as an affine lattice
whenever we are free to do so. In particular $P$ has dimension
$d$, and we need not distinguish \emph{normal} and
\emph{integrally closed} lattice polytopes (in the terminology
of \cite{BG}).

The set $E(P)=P\times\{1\}\cap \ZZ^{d+1}$ generates a submonoid
$M(P)$ of $\ZZ^{d+1}$. The monoid algebra $R=K[P]=K[M(P)]$ over
a field $K$ has a natural grading and is generated by monomials
of degree $1$ (that, by construction, correspond to the lattice
points of $P$). Thus $\cV=\Proj R$ is a projective subvariety
of the projective space $\PP_K^n$, $n=\#E(P)-1$. The natural
affine charts of $\cV$ are the spectra of monoid algebras
obtained by dehomogenizing $R$ with respect to the monomials
that correspond to the vertices of $P$: for such a vertex $v$
the corresponding chart is given by $\Spec K[M(P)_v]$ where the
monoid $M(P)_v\subset \ZZ^{d}$ is generated by the difference
vectors $x-v$, $x\in P\cap L$. (In multiplicative notation,
$x-v$ corresponds to the quotient of two monomials of degree
$1$.)

In classical terminology, a toric variety is required to be
normal, and this is the case if and only if all the algebras
$K[M(P)_v]$ are normal. A polytope $P$ with this property is
called \emph{very ample} since it represents a very ample line
bundle on $\cV$. Under our assumptions on $P$, normality of
$K[M(P)_v]$  is equivalent to the equality $M(P)_v=C(P)_c\cap
L$ where $C(P)_v$ is the \emph{corner} (or \emph{tangent}) cone
of $P$ at $v$: it is generated by the vectors $x-v$, $x\in P$.

The variety $\cV$ is smooth if and only if the algebras
$K[M(P)_v]$ are polynomial rings. In terms of $P$, this
property can be characterized as follows: exactly $d$ edges
emanate from each vertex $v$, and the $d$ vectors $w-v$, where
$w$ is the lattice point next to $v$ on an edge, are a basis of
$L$. Such polytopes are called \emph{smooth}.

The combinatorial approach to the open questions (N) and (Q) is
justified by the fact that the homogeneous coordinate rings
considered in these questions are all of type $K[P]$ where $P$
is a smooth polytope.

The rather elementary passage from a lattice polytope $P$ to a
projective toric variety has been sketched above. In order to
justify the last claim we have to reverse the construction.
Every projective toric variety of dimension $d$ intrinsically
defines a complete projective fan $\cF$ in $V^*$ (we remind the
reader on our convention on primal and dual space), and the
coordinate rings of the equivariant embeddings of $\cV$ into
projective space are given in the form $\Proj K[P]$ where $P$
is a very ample lattice polytope (satisfying all our basic
assumptions) such that $\cN(P)=\cF$. Such a polytope is called
a \emph{support polytope} of $\cF$. The corner cones of $P$ are
exactly the cones dual to the cones in $\cF^{[d]}$ (the set of
$d$-dimensional faces of $\cF$), and since duality preserves
unimodularity, $\cV$ is smooth if and only if $\cF$ is
unimodular.

The correspondence between unimodular projective fans and
smooth projective toric varieties is not only of fundamental
theoretical importance---it offers a way to construct smooth
polytopes from ``random data''.

\section{Polytopes from fans}\label{Fans}

\subsection{Creating unimodular projective fans}
As just explained, the choice of a smooth projective
variety is equivalent to the construction of a projective
unimodular fan. It can be carried out as follows:
\begin{enumerate}
\item[(UF1)] Choose vectors $\rho_1,\dots,\rho_s$ in $L^*$
    such that the origin is in the interior of
    $Q=\conv(\rho_1,\dots,\rho_s)$. The cones spanned by
    the faces of $Q$ (with apex in $0$) then form a
    projective fan $\cF$.

\item[(UF2)] Choose a regular triangulation of the boundary
    of $Q$ with vertices in lattice points, and replace
    $\cF$ by the induced simplicial refinement.

\item[(UF3)] For each maximal cone of $\cF$ compute its
    Hilbert basis and refine $\cF$ by stellar subdivision,
    inserting all these vectors in some random order.

\item[(UF4)] Repeat (UF3) until a unimodular fan is
    reached.
\end{enumerate}
This is nothing but the algorithm producing an equivariant
desingularization of the projective toric variety defined by
the choice of $\cF$ in (UF1). It terminates in finitely many
steps since stellar subdivision by Hilbert basis elements
strictly reduces the multiplicities of the simplicial cones.
Computing unimodular fans is fast, contrary to the computation
of support polytopes.

\subsection{Computing support polytopes}
Once we have a unimodular fan $\cF$, we must find support
polytopes of $\cF$. Let us first assume that $\cF$ is just an
arbitrary complete fan. The algorithm that we describe in the
following will decide whether $\cF$ is projective by providing
lattice polytopes $P$ such that $\cN(P)=\cF$ in the projective
case, and ending with a negative outcome otherwise. The set of
rays of $\cF$ is denoted by $\cF^{[1]}$, and its elements are
listed as $\rho_1,\dots,\rho_s$.

Each support polytope is the solution set of a system of linear
inequalities
$$
\rho_i(x)\ge -b_i,\qquad b_i\in\ZZ,\ i=1,\dots,s.
$$
(The choice of the minus sign will turn out natural.) We are
searching for the right hand sides $b=(b_1,\dots,b_s)\in
\cW=\ZZ^s$, such that the following conditions are satisfied:
\begin{enumerate}
\item[(LP)] For each cone $\Sigma\in\cF^{[d]}$ there exists
    a vector $v_\Sigma\in L$ such that
    $\rho_i(v_\Sigma)=-b_i$ for $\rho_i\in \Sigma$.
\item[(CP)] The points $v_\Sigma$ are indeed the vertices
    of their convex hull $P(b)$.
\item[(VA)] $P(b)$ is very ample.
\end{enumerate}

In fact, for each $\Sigma\in\cF^{[d]}$, the hyperplanes with
equations $\rho_i(x)=-b_i$, $\rho_i\in\Sigma$, must meet in a
lattice point $v_\Sigma\in L$. Thus condition (LP) selects a
sublattice $\cC$ of $\cW$, and there is a well-defined linear
map $\vert_\Sigma: \cC\to L$ that assigns $b\in \cC$ the
prospective vertex $v_\Sigma$.

In the unimodular case, (LP) is satisfied for all
$b\in\cW=\cC$, and this simplifies the situation significantly!
Also (VA) is automatically satisfied. Therefore we concentrate
on condition (CP) which requires that the points $v_\Sigma$ are
in convex position. This is equivalent to the system
\begin{equation}
\rho_j(\vert_\Sigma(b)) \ge -b_j+1,\qquad j=1,\dots,s,\
\rho_j\notin \Sigma,\quad \Sigma \in \cF.\label{veryample}
\end{equation}
of inequalities being satisfied. Convexity is a local
condition, and therefore one can restrict the system to a
smaller set of inequalities: one needs to consider only the
inequalities $\rho_j(\vert_\Sigma(b)) > -b_j$ such that
$\rho_j\notin\Sigma$ is a ray in a cone $T\in\cF$ sharing a
facet with $\Sigma$. (This condition is easily checked
algorithmically.) The set of pairs $(\Sigma,j)$ just defined is
denoted by $\cS$.

We summarize and slightly reformulate our discussion as
follows. Set
$$
N'=\{b\in\RR^s: \rho_j(\vert_\Sigma(b))+ b_j\ge 0,\ (\Sigma,j)\in\cS\}.
$$
Then the lattice polytopes we are trying to find correspond to
the points in $\cC\cap\relint(N')$. The cone $N'$ is not
pointed since it contains a copy of $V$, namely the vectors
$(\rho_j(v))$. However, we loose nothing if we choose a cone
$\Sigma_0\in\cF^{[d]}$ and set $v_{\Sigma_0}=0$. In this way we
intersect $N'$ with a linear subspace $U$, and the intersection
$$
N=N'\cap U
$$
is indeed pointed. Moreover, since $0$ is one vertex of the
polytopes to be found, we have $b_j\ge 0$ for all $j$ and all
points $b\in N'$. In particular this implies
$$
P(b)\subset P(b+b')
$$
for all $b,b'\in N\cap\cC$.

Heuristically, and for the reasons pointed out above, the
candidates for counterexamples should appear among the
\emph{inclusion minimal} polytopes $P$ with $\cF=\cN(P)$. In
view of the inclusion just established, it is enough to
determine the minimal system of generators of $\cC\cap\relint
N$ as an ideal of the monoid $\cC \cap N$. After homogenization
of the system \eqref{veryample} and fixing
$\vert_{\Sigma_0}=0$, this amounts to a Hilbert basis
calculation in the cone $\tilde N\subset \RR^{s-d+1}$ defined
by the inequalities
\begin{align*}
\rho_j(\vert_\Sigma(b))+b_j-h&\ge 0,\qquad (\Sigma,j)\in\cS,\\
h&\ge 0.
\end{align*}
From the Hilbert basis computed we extract the elements with
$h=1$, and obtain a collection of polytopes among which we
easily find the inclusion minimal ones. (If no such element
exists, the fan has proved to be non-projective.)

\begin{remark}
(a) The letters $\cW$ and $\cC$ have been chosen since $\cW$
represents the group of torus invariant Weyl divisors and $\cC$
the group of torus invariant Cartier divisors. By fixing one of
the vertices at the origin, we have chosen a splitting of the
epimorphism $\cC\to \Pic(\cV)$.
\end{remark}

\begin{remark} \label{Pract}
(a) Since we favor small examples and since very large ones
tend to be intractable, we limit the construction of unimodular
fans to at most $d+20$ rays (in other words, $\rank
\Pic(\cV)\le 20$) and $150$ maximal simplicial cones. These
numbers can be varied, of course, but they allow the
computation of support polytopes.

(b) Computing all minimal support polytopes is only possible if
the rank of the Picard group is not too large. In practice,
$10$ is a reasonable bound. One way out is to compute the
extreme rays of $\widetilde N$ and to select those with $h=1$.
Though there is no guarantee for the existence of such extreme
rays, they almost always exist.

(c) Not every element of the Hilbert basis computed defines an
inclusion minimal support polytope of the given fan. The reason
is that the difference of the corresponding support functions
of the fan need not be convex.
\end{remark}

\section{Normality}\label{Norm}

\subsection{Checking normality} For this purpose we use the
Hilbert basis algorithm of Normaliz (in the author's
experimental library). We refer to \cite{BK} and \cite{BI} for
the details. The algorithm is fast for the polytopes that have
been investigated.

\subsection{Extending corner covers}\label{CornExt}
One reason for which one could expect smooth polytopes $P$ to
be normal is that each corner is covered by unimodular
simplices---in fact, it is such a simplex---and that it should
be possible to extend these covers far enough into $P$ such
that $P$ is covered by the extensions. In order to define the
extensions one can identify the corner cone $C(P)_v$ with the
cone $D$ generated by the sums of unit vectors
$e_1,e_1+e_2,\dots,e_1+\dots+e_d$ in $\RR^d$ and transfer the
Knudsen-Mumford triangulation of $D$ to a triangulation
$\Sigma$ of $C(P)_v$ (compare \cite[Ch.~3]{BG}. Then we let
$U_v(P)$ be the union of those unimodular simplices of $\Sigma$
that lie in $P$. If $P=\bigcup_v U_v(P)$, then $P$ is evidently
normal. It would be possible to test whether $P=\bigcup_v
U_v(P)$. However, a direct test for normality is much faster.
We will come back to this point in \ref{superconn} below.
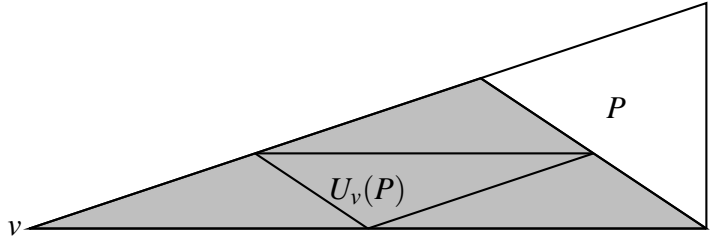
$$
\psset{xunit=1.0cm, yunit=2.0cm}
\begin{pspicture}(0,-0.5)(9,1.5)
 \pspolygon[fillstyle=solid, fillcolor=lightgray](0,0)(9,0)(6,1)
 \psline(0,0)(9,0)
 \psline(0,0)(9,1.5)(9,0)
 \psline(3,0.5)(7.5,0.5)
 \psline(4.5,0)(7.5,0.5)
 \psline(3,0.5)(4.5,0)
 \psline(6,1)(9,0)
 \rput(-0.2,0.0){$v$}
 \rput(7.8,0.8){$P$}
 \rput(4.5,0.25){$U_v(P)$}
\end{pspicture}
$$

\section{Generation in degree 2}\label{Quadr}

In the following we must work with the toric ideal of a
polytope and its dehomogenizations, and some precise notation
is needed. The monoid algebra $K[M(P)]$ has a natural
presentation as a residue class ring of a polynomial ring
$$
A_P=K[X_x:x\in P\cap\ZZ^d]\xrightarrow{\phi} K[M(P)],\qquad
X_x\mapsto (x,1)\in M(P)\subset K[M(P)].
$$
The kernel of $\phi$ is the toric ideal of $P$. It is generated
by all binomials
$$
\prod_x X_x^{a_x}-\prod_x X_x^{b_x}\qquad\text{such that}\qquad
\sum_xa_x x=\sum_x b_x x.
$$
The lattice points $w$ next to a vertex $v$ of $P$ on an edge
of $P$ will be called \emph{neighbors} of $v$.

\subsection{Testing generation in degree $2$}\label{Deg2Deg3}
There seems to be no other way for testing generation in degree
$2$ than computing the toric ideal via a Gröbner basis method.
(For toric ideals, special algorithms have been devised; see
\cite{HS} or \cite{BR}.) In order to access also rather large
polytopes we only test whether the toric ideal $I(P)$ needs
generators in degree $3$ as follows, letting $J$ denote the
ideal generated by the degree $2$ binomials in $I(P)$:

\begin{enumerate}
\item[(GB2)] We compute the degree $2$ component $G$ of a
    Gröbner basis of $J$ with respect to a reverse
    lexicographic term order by simply scanning all degree
    $2$ binomials in $I(P)$.

\item[(GB3)] Next we compute the degree $3$ component of
    the Gröbner basis of $J$ by the Buchberger algorithm in
    a specialized data structure.

\item[(HV3)] Then we compute the $h$-vector of $J$ up to
    degree $3$, using the initial ideal from the preceding
    steps.

\item[(NMZ)] Finally the result is compared to the
    $h$-vector of $K[P]$ computed by Normaliz. (This
    $h$-vector gives the numerator polynomial of the
    Ehrhart series of $P$.)
\end{enumerate}
Clearly, $I(P)$ has no generators in degree $3$ if and only if
the two $h$-vectors coincide up to hat degree. The time
consuming step is (GR3) while the others are very fast.

As we will see next, there is a generalization of generation in
degree $2$ that is very natural from the viewpoint of
projective geometry and much faster to decide.

\subsection{Scheme-theoretic generation in degree $2$} Let
$I\subset A=K[X_1,\dots,X_n]$ be a homogeneous ideal (with
respect to the standard grading). We say that $I$ is
\emph{scheme-theoretically generated} in degree $2$ if there
exists an ideal $J$ that is generated by homogeneous elements
of degree $2$ such that $I$ and $J$ have the same saturation
with respect to the maximal ideal $\mm=(X_1,\dots,X_n)$:
$$
I^{\textrm{sat}}=\{x:\mm^kx\in I\text{ for some $k$}\}=J^{\textrm{sat}}
$$
Equivalently, we can require that $I$ and $J$ define the same
ideal sheaf on $\Proj A$. Clearly $J\subset I$ if
$I=I^{\textrm{sat}}$, and this is the case for prime ideals $I$
like toric ideals, whence we may assume that $J=I_{(2)}$ is the
ideal generated by the degree $2$ elements of $I$.

Let $v$ be a vertex of $P$, $R=K[M(P)]$ and $S=R/(X_v-1)$ be
the dehomogenization of $R$ with respect to $X_v$. The
presentation $K[M(P)]=A_P/I(P)$ induces a presentation
$$
B_v=A_P/(X_v-1)\to S\quad \text{with kernel}\quad I(P)B_v.
$$
The residue classes of the $X_x$ are denoted by $Y_x$, and
$B_v$ is again a polynomial ring in the $Y_x$, $x\neq v$. But
$S$ is already the image of the subalgebra $B_v'=K[Y_x:x-v\in
H_v]$, $H_v=\Hilb(M(P)_v)$. Clearly $I_v$ is generated by the
extension of the toric ideal $I(H_v)=I(P)B_v\cap B_v'$ and any
choice of polynomials $Y_z-\mu_z$ where $z-v\notin H_v$ and
$\mu_z\in B_v'$ is a monomial representing $z-v$ as a
$\ZZ_+$-linear combination of the Hilbert basis $H_v$. The
smooth case is characterized by the condition  $I(H_v)=0$. This
simplifies the situation considerably. For the proof of the
next theorem one should note that $I(P)B_v=JB_v$ for a
homogeneous ideal $J$ in $A_P$ if and only if
$I(P)[X_v^{-1}]=J[X_v^{-1}]$.

In order to prove that $I(P)$ is the saturation of $I(P)_{(2)}$
have also to check the dehomogenization with respect to
indeterminates $X_x$ for non-vertices $x$ of $P$. However, as
we will see, the comparison can be reduced to the consideration
of vertices.\pagebreak[3]

\begin{theorem}\label{scheme2}
Let $P$ be a smooth lattice polytope. Then the following are
equivalent:
\begin{enumerate}
\item $I(P)$ is scheme-theoretically generated in degree
    $2$;

\item \begin{enumerate}

\item for all vertices $v$ of $P$ and all lattice points
    $x$ that are not neighbors of $v$ the polytope $P\cap
    (x+v-P)$ contains a lattice point $y\neq v,x$, and
\item every non-vertex lattice point $x$ of $P$ is the
    midpoint of a line segment $[y,z]$, $y,z\in P\cap \ZZ^d$,
    $y,z\neq x$.
\end{enumerate}
\end{enumerate}
\end{theorem}
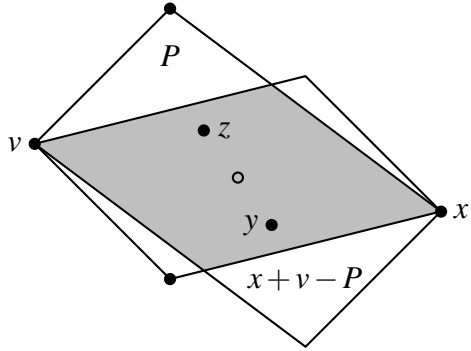
$$
\psset{unit=0.9cm}
\def\vertex{\pscircle[fillstyle=solid,fillcolor=black]{0.09}}
\def\overtex{\pscircle{0.09}}
\begin{pspicture}(0,-1)(6,4.5)
 \pspolygon[linecolor=white,fillstyle=solid,fillcolor=lightgray](0,2)(3.5,2.875)(6,1)(2.5,0.125)
 \rput(0,2){\vertex}
 \rput(2,0){\vertex}
 \rput(2,4){\vertex}
 \rput(6,1){\vertex}
 \pspolygon(0,2)(2,0)(6,1)(2,4)
 \pspolygon[](0,2)(4,-1)(6,1)(4,3)
 \rput(-0.3,2){$v$}
 \rput(6.3,1){$x$}
 \rput(3.5,0.8){\vertex}
 \rput(3.2,0.8){$y$}
  \rput(3,1.5){\overtex}
 \rput(2.5,2.2){\vertex}
 \rput(2.8,2.2){$z$}
 \rput(2,3.3){$P$}
 \rput(4.0,-0.0){$x+v-P$}
\end{pspicture}
$$

\begin{proof}
For the implication (2) $\implies$ (1) we note that for every
vertex and each non-neighbor $x$ we have a binomial
$$
X_vX_x-X_yX_z\in I(P)
$$
by (2)(a). Thus $Y_x-Y_yY_z\in I(P)_{(2)}B_v$. As a positive monoid,
$M(P)_v$ has a grading, and we can use induction on degree to find
monomials $\mu_x$ in the $Y_w$, $w-v\in H_v$, such that
$Y_x-\mu_x\in I(P)_{(2)}B_v$. In fact, $\deg (y-v),\deg (z-v)<\deg
(x-v)$, so that $Y_y-\mu_y$, $Y_z-\mu_z\in I(P)_{(2)}B_v$. Then
$Y_x-\mu_y\mu_z\in I(P)_{(2)}B_v$ as well, and we can set
$\mu_x=\mu_y\mu_z$. This argument shows that $I(P)B_v=I(P)_{(2)}B_v$
or, equivalently, $I(P)[X_v^{-1}]=I(P)_{(2)}[X_v^{-1}]$ for all
vertices $v$ of $P$.

It remains to compare $I(P)$ and $I(P)_{(2)}$ after the inversion of
$X_p$ for non-vertices $p$. Let $Q$ be the convex hull of all
lattice points $w$ such that $X_w$ is a unit modulo
$I(P)_{(2)}[X_p^{-1}]$. One has $Q\neq\emptyset$ since $p\in Q$. Let
$x$ be a vertex of $Q$. If $x$ is a vertex of $P$, we can invert
$X_x$ first and use what has been shown above as a consequence of
(2)(a). Suppose that $x$ is a non-vertex of $P$. Then (2)(b) implies
that $X_x^2-X_yX_z\in I(P)$ for suitable lattice points $y,z$. But
only one of $y,z$ can belong to $Q$, and both are units together
with $X_x$ after the inversion of $X_p$. This is a contradiction,
and thus $x$ is a vertex of $P$.

For (1) $\implies$ (2)(a) we consider the binomial $Y_x-\mu_x\in
I(P)B_v$. Homogenizing and clearing denominators with respect to
$X_v$ yields a binomial
$$
\beta=X_v^kX_x-X_v^p\prod X_w^{a_w}\in I(P).
$$
Multiplying by a high power of $X_v$ sends it into $(P)_{(2)}$.
So we may assume that $\beta$ belongs to $(P)_{(2)}$. But then
$X_v^kX_x$ must be divisible by a monomial appearing in a
degree $2$ binomial in $I(P)$. The only potential degree $2$
divisors are $X_v^2$ and $X_vX_x$. Since $v$ is a vertex, no
power $X_v^m$, $m>1$, can appear in a binomial $\gamma$ in
$I(P)$ as one of the summands (unless $\gamma$ is divisible by
$X_v$). This implies that there exists a degree $2$ binomial
$X_vX_x-X_yX_z$ in $I(P)$, and $y$ and $z$ both belong to
$P\cap(x+v-P)$ since they both belong to $P$.

The argument for (b) is similar. In fact, let $F$ be the smallest
face of $P$ containing $x$. Since $x$ is a non-vertex, $F$ must
contain at least one more lattice point $w$. Modulo $I(P)[X_x^{-1}]$
all $X_w$ for lattice points $w$ in $F$ are units. Since
$I(P)[X_x^{-1}]=I(P)_{(2)}[X_x^{-1}]$ by hypothesis, the same holds
modulo $I(P)_{(2)}[X_x^{-1}]$. By similar arguments as above this
implies the existence of a binomial
$$
X_x^k-X_w\mu_w
$$
in $I(P)_{(2)}$. But then we must have a nonzero binomial
$X_x^2-X_yX_z\in I(P)$.
\end{proof}

\begin{remark}\label{deg2gen}
The implication (2) $\implies$ (1) can be generalized as follows:
smoothness is replaced by the hypothesis that  for each vertex $v$
the ideal $I(H_v)$ is generated by homogeneous binomials of degree
$2$. In fact, \emph{homogeneous} binomials in $I(H_v)$ lift to
homogeneous binomials in $I(P)$, and we need only add the binomials
$X_vX_x-X_yX_z$ and $X^2-X_yX_z$ existing by (2) in order to find a
degree $2$ subideal $J$ of $I(P)$ such that $JB_x=I(P)B_x$ for all
lattice points $x\in P$.

The implication (1) $\implies$ (2) holds for arbitrary lattice
polytopes if one restricts (2)(a) to those $x$ for which
$x-v\notin H_v$. However, one cannot conclude that the ideals
$I(H_v)$ are generated in degree $2$. We will come back to this
point in Remark \ref{matroid}.
\end{remark}

By condition (2) of the theorem, scheme-theoretic generation in
degree $2$ can be tested very quickly. In most cases $y$ can be
chosen as a neighbor of $v$ for (2)(a). Thus the number of
lattice points $y\in P$ to be tested is very small on average.
However, the theorem very strongly indicates that finding a
counterexample to scheme-theoretic generation in degree $2$ is
extremely difficult.

\begin{remark}
In the author's first formulation of Theorem \ref{scheme2}
condition (2)(b) was missing. This mistake was pointed out by
Mateusz Michalek who also suggested the inclusion of (2)(b).
The problem is discussed in \cite[Remark 4.2]{Mich}.
\end{remark}

\subsection{Abundant degree $2$ relations}
All smooth polytopes that have come up in the search for
counterexamples satisfy an even stronger condition: let us say
that $P$ has \emph{abundant degree $2$ relations} if condition
(2)(a) continues to hold if we replace the vertex $v$ by an
arbitrary lattice point: for all lattice points $v,x\in P$
(including the case $v=x$) the midpoint of the segment $[v,x]$
is also the midpoint of a different line segment $[y,z]\subset
P$, apart from the following obvious exceptions: one of $v,x$
is a vertex, say $v$, and $x=v$ or $x$ is a neighbor of $v$. In
terms of the toric ideal $I(P)$: it contains a binomial
$X_vX_x-X_yX_z\neq 0$ for all lattice points $v,x$ unless this
is priori impossible. (This includes condition (2)(b).)

The property of having abundant degree $2$ relations clearly
implies scheme-theoretic generation in degree $2$ for smooth
polytopes, but it is not clear how it is related to generation
in degree $2$. For arbitrary lattice polytopes it does not
follow from generation in degree $2$. As an example one can
take the join of two line segments with midpoints whose toric
ideal is generated by $X_xX_z-X_y^2$, $X_uX_w-X_v^2$: the
midpoint of $[y,v]$, both non-vertices, is not the midpoint of
any other line segment since $X_yX_v$ does not appear in the
generating binomials.

One is tempted to prove that abundant degree $2$ relations
imply generation in degree $2$ by  a Gröbner basis argument.
While we cannot exclude that such an argument is possible, it
is very clear that its success depends on the choice of the
term order. A simple example is the following polytope with the
term order that induces the unimodular triangulation: the
corresponding Gröbner basis contains $X_xX_yX_z-X_w^3$.
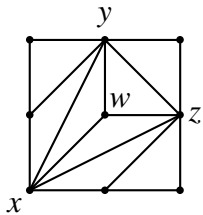
$$
\begin{pspicture}(0,0)(2,2.5)
 \multirput(0,0)(1,0){3}{\vertex}
 \multirput(0,1)(1,0){3}{\vertex}
 \multirput(0,2)(1,0){3}{\vertex}
 \pspolygon(0,0)(2,0)(2,2)(0,2)
 \psline(0,0)(1,2)
 \psline(0,0)(1,1)
 \psline(0,0)(2,1)
 \psline(0,1)(1,2)
 \psline(1,0)(2,1)
 \psline(1,1)(1,2)
 \psline(1,1)(2,1)
 \psline(1,2)(2,1)
 \rput(-0.2,-0.2){$x$}
 \rput(1.0,2.3){$y$}
 \rput(2.2,1.0){$z$}
 \rput(1.2,1.2){$w$}
\end{pspicture}
$$

\begin{remark}\label{matroid}
An important class of polytopes that have abundant degree $2$
relations, but are not known to be quadratically defined, are
given by matroids. See \cite{Schw} for a rather recent result
and references for this very hard problem.

In fact, it is not difficult to see that the symmetric exchange
in matroids supplies abundant degree $2$ relations. So one
could try to apply the generalization \ref{deg2gen} of Theorem
\ref{scheme2}, (2) $\implies$ (1), in order to show
scheme-theoretic generation in degree $2$ for matroid
polytopes. But this does not work since the ideal $I(H_v)$ need
not be generated in degree $2$, even if $I(P)$ is.

An example for this phenomenon is given by the matroid whose
bases are the triples $[i_1,i_2,i_3]$ indicating the vector
space bases contained in the family $u_i=e_i$, and
$u_{i+3}=e_j+e_k$, $i=1,2,3$, $1\le j<k\le 3$, $i\neq j,k$. For
the vertex $v=[1,2,3]$ the ideal $I(H_v)$ is generated by a
degree $3$ binomial.
\end{remark}

\subsection{Squarefree divisor complexes} Whether the toric
ideal $I(P)$ needs generators in a certain multidegree can be
tested by checking the connectivity of the squarefree divisor
complex of the given degree. For example, see \cite{BH} for the
terminology and the details. (According to Stanley
\cite[7.9]{St}, the result goes back to Hochster.) Such a test
has been implemented and is applied to a random selection of
multidegrees of total degree $3$ for polytopes that are too
large for the approach described in \ref{Deg2Deg3}. However, it
seems rather hopeless to find a counterexample to quadratic
generation by this test since it can only deal with a single
multidegree at a time.

\section{Chiseling}\label{Chisel}

As pointed out above, one should expect counterexamples to be
small, but ``complicated''. In particular, one can try to pass
from a smooth polytope $P$ to a smooth subpolytope by splitting
$P$ along a suitable hyperplane. In its simplest form this
amounts to cutting corners off $P$ as illustrated in the figure
below. The operation described in the following is called
\emph{chiseling}---it sounds less cruel than ``cutting faces''.
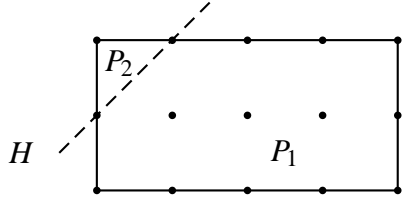
$$
\begin{pspicture}(-1,0)(4,3)
 \multirput(0,0)(1,0){5}{\vertex}
 \multirput(0,1)(1,0){5}{\vertex}
 \multirput(0,2)(1,0){5}{\vertex}
 \pspolygon(0,0)(4,0)(4,2)(0,2)
 \psline[linestyle=dashed](-0.5,0.5)(1.5,2.5)
 \rput(2.5,0.5){$P_1$}
 \rput(0.3,1.7){$P_2$}
 \rput(-1,0.5){$H$}
\end{pspicture}
$$

Suppose $F$ is a face of the smooth polytope $P$ and $\aff(P)$
the intersection of the support hyperplanes $H_1,\dots,H_s$
where $s=d-\dim F$. Let $H_i$ be given by the equation
$\lambda_i(x)=b_i$ (and $\lambda_i(x)\ge b_i$ on $P$). Set
$\sigma=\lambda_1+\dots+\lambda_s$. Then $\sigma$ has constant
value $b=b_1+\dots+b_s$ on $F$. Let $c$ be the minimum value of
$\sigma$ on the vertices $x$ of $P$ that do not lie in $F$, and
suppose that $b<c-1$ (clearly $b<c$). In this case the
hyperplane $H$ with equation $\sigma(x)=c-1$ splits $P$ into
the polytopes
\begin{equation}
\begin{aligned}
P_1&=\{x:x\in P\text{ and }\sigma(x)\ge c-1\},\\
P_2&=\{x:x\in P\text{ and }\sigma(x)\le c-1\}.
\end{aligned}\label{split}
\end{equation}
\begin{lemma}\label{Hsmooth}
$P_1$ and $P_2$ are smooth lattice polytopes of full dimension.
\end{lemma}

This is easily seen by considering the normal fans of $P_1$ and
$P_2$. The normal fan of $P_1$ is a stellar subdivision of
$\cN(P)$. The next theorem shows that it suffices to
investigate $P_1$ and $P_2$ in the search for counterexamples.
As in \cite{BG} we say that a polytope is \emph{integrally
closed} if $P\cap \ZZ^d$ generates $\ZZ^d$ affinely and the
monoid ring $K[P]$ is normal, or, equivalently, $K[P]$ is
integrally closed in $K[\ZZ^{d+1}]$.

\begin{theorem}\label{robust}
Let $P$ be a lattice polytope and $H$ a rational hyperplane
such that $P_1=P\cap H^+$ and $P_2=P\cap H^-$  are lattice
polytopes.
\begin{enumerate}
\item If $P_1$ and $P_2$ are integrally closed, then $P$ is
    integrally closed.
\item If $P_1$ and $P_2$ are integrally closed and the
    toric ideals of $P_1$ and $P_2$ are generated in
    degrees $d_1$ and $d_2$ respectively, then the toric
    ideal of $P$ is generated in degrees
    $\le\max(2,d_1,d_2)$.
\end{enumerate}
\end{theorem}

\begin{proof}
(1) is obvious.

(2) We define a weight function on the lattice points $x$ of
$P$ (or the generators of $M(P)$) by $w(x)=|\lambda(x)|$ where
$\lambda$ is the primitive integral affine linear form defining
$H$ by the equation $\lambda(x)=0$. This weight ``breaks'' $P$
along $H$, and it breaks the monoid $M(P)$into the monoidal
complex consisting of $M(P_1)$ and $M(P_2)$, glued along $H$.
We refer the reader to \cite[Chapter 7]{BG} for the terminology
just used.

Let $I$ be the toric ideal of $P$. The normality of $P_1$ and
$P_2$ implies that $M(P)\cap H^+=M(P_1)$ and $M(P)\cap
H^-=M(P_2)$. By \cite[Corollary 7.19]{BG} this is equivalent to
the fact that $\ini_w(I)$ is a radical ideal.

Therefore $\ini_w(P)$ is the defining ideal of the monoidal
complex by \cite[Theorem 7.18]{BG}. On the other hand, the
defining ideal of the monoidal complex is generated by the
binomial toric ideals of $P_1$ and $P_2$ and the monomial ideal
representing the subdivision of $P$ along $H$ \cite[Proposition
7.12]{BG}. The latter is of degree $2$ since a monomial with
support not in $P_1$ or $P_2$ must have a factor of degree $2$
with this property (as is always the case by subdivisions along
hyperplane arrangements). This shows that $\ini_w(I)$ is
generated in degrees $\le\max(2,d_1,d_2)$ and it follows that
$I$ itself is generated in degrees $\le \max(2,d_1,d_2)$.
\end{proof}

For a special case, Theorem \ref{robust} is contained in an
unpublished manuscript of Gubeladze and Ho\c{s}ten, but with a
proof using squarefree divisor complexes.

The following counterexample shows that part (1) of the theorem
cannot be reversed, and that part (2) does no longer hold if
one omits the assumption that $P_1$ and $P_2$ are integrally
closed: set
\begin{gather*}
x=(0,0,0),\qquad y=(1,0,0),\qquad z=(0,1,0),\\
v=(1,1,2),\qquad w=(0,0,-1),
\end{gather*}
and let $P$ be the polytope spanned by them. One can easily
check that the given points are the only lattice points of $P$
since $P$ is the union of the simplices $P_1=\conv(x,y,z,v)$
and $P_2=\conv(x,y,z,w)$. The toric ideals of $P_1$ and $P_2$
are both $0$, but $I(P)$ is generated by the binomial
$X_xX_yX_z-X_vX_w^2$. Condition (a) is violated since the
lattice points in $P_1$ do not span $\ZZ^2$ as an affine
lattice.

Let us say that a lattice polytope is \emph{robust} if it
cannot be chiseled into two lattice subpolytopes of the same
dimension along a hyperplane $H$, as described by
\eqref{split}. The robust smooth polytopes $P$ can be
characterized as follows: from each face $F$ of $P$ there
emanates an edge of length~$1$.

\begin{corollary}
\begin{enumerate}
\item In every dimension, a minimal counterexample to the
    normality question is robust.
\item In every dimension, a normal counterexample to
    generation in degree $2$ is robust.
\end{enumerate}
\end{corollary}

Chiseling has been used in two ways: (1) to reduce the size of
polytopes that had been computed from unimodular fans, and (2)
to produce polytopes from large smooth polytopes by chiseling
them parallel to faces chosen in some random order.

\begin{remark}
In our search, we only investigate $P_1$ further, although we
cannot exclude that ``bad'' properties of $P$ come from $P_2$.
Neglecting $P_2$ is however justified if the face $F$ is a
vertex. In this case $P_2$ is a multiple of the unit simplex
and $I(P_2)$ is generated in degree $2$ since $K[M(P_2)]$ is a
Veronese subalgebra of a polynomial ring.
\end{remark}

\section{Miscellaneous properties of smooth polytopes}
\subsection{Superconnectivity and strong connectivity}
\label{superconn} Let $v$ be a vertex of $P$. A
\emph{$\Hilb_v$-path} is a sequence of lattice points
$v=x_0,x_1,\dots,x_m$ in $P$ such that $x_{i+1}-x_i\in
\Hilb(M(P)_v)$ for all $i=0,\dots,m-1$. We say that $P$ is
\emph{superconnected} if every lattice point in $P$ is
connected to every vertex $v$ by a $\Hilb_v$-path. However,
superconnectivity is rarely satisfied, and even in dimension
$2$ one easily finds counterexamples. Consider the smooth
polygon with vertices $(0,0)$, $(4,0)$, $(4,1)$, $(3,3)$,
$(2,4)$, and $(0,4)$:
$$
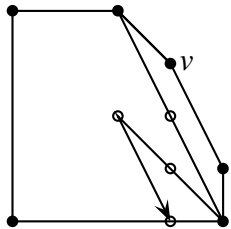
\psset{unit=0.7cm,arrowsize=2.5pt 4}
\def\vertex{\pscircle[fillstyle=solid,fillcolor=black]{0.11}}
\def\overtex{\pscircle{0.11}}
\begin{pspicture}(0,0)(4,4)
 \pspolygon(0,0)(4,0)(4,1)(3,3)(2,4)(0,4)
 \rput(0,0){\vertex}
 \rput(4,0){\vertex}
 \rput(4,1){\vertex}
 \rput(3,3){\vertex}
 \rput(3.3,3){$v$}
 \rput(2,4){\vertex}
 \rput(0,4){\vertex}
 \psline{->}(3,3)(2,4)(3,2)(4,0)(3,1)(2,2)(3,0)
 \rput(3,2){\overtex}
 \rput(3,1){\overtex}
 \rput(2,2){\overtex}
 \rput(3,0){\overtex}
\end{pspicture}
$$
The lattice point $(3,0)$ is reachable by a $\Hilb_v$-path from
$v=(3,3)$, but the origin is not.

On the other hand, every smooth polytope encountered in the
search satisfies the following weaker condition: for each
lattice point $x$ there exists at least one vertex $v$ to which
$x$ is connected by a $\Hilb_v$-path. We call such polytopes
\emph{strongly connected}. If strong connectivity should fail
for a smooth polytope $P$, then $P$ has a lattice point
$x\notin\bigcup_v U_v(P)$, and in particular $P\neq\bigcup_v
U_v(P)$ (compare \ref{CornExt}).

Superconnectivity of $P$ is equivalent to the following
condition: the lattice point $y$ in condition (2)(a) of Theorem
\ref{scheme2} can always be chosen among the neighbors of $v$.
Therefore superconnectivity can be considered to be a strong
form scheme-theoretic generation in degree $2$ for
\emph{smooth} polytopes without non-vertex lattice points.

The matroid polytopes discussed briefly in Remark \ref{matroid}
are superconnected since the symmetric exchange of single
elements produces neighbors.

\subsection{Positivity of coefficients of the Ehrhart
polynomial} Since we are computing the Ehrhart series of smooth
polytopes anyway, checking the Ehrhart polynomial for
positivity of its coefficients costs no extra time. In fact,
for all smooth polytopes found in our search these coefficients
proved to be positive. Therefore it seems reasonable to ask the
following question:

\begin{question}
Do the Ehrhart polynomials of smooth polytopes have positive
coefficients?
\end{question}

See De Loera, Haws and Köppe \cite{DHK} for a discussion of the
positivity question for another class of polytopes.

\section{Very ample polytopes}\label{VeryA}

Smooth polytopes can be considered as special instances of the
following class:
\begin{enumerate}
\item[(HC)] for each vertex $v$ of the (simple) polytope
    $P$ the points $x$ such that $x-v\in\Hilb(M(P)_v)$ are
    contained in the polytope spanned by $v$ and its
    neighbors.
\end{enumerate}
Such polytopes $P$ are automatically very ample, provided the
lattice points in $P$ span $\ZZ^d$ affinely. (As usual, a
polytope is \emph{simple} if exactly $d$ hyperplanes meet in
each of its vertices.)

Generalizing question (N), one may ask whether polytopes of
class (HC) are normal. As the counterexample below shows, the
answer is ``no'' for non-simple polytopes, but seems to be
unknown for simple ones. In fact, we do not know of any simple
very ample, but non-normal polytope.

The question (HC) fits into a line of research that relates
normality of polytopes to the length of their edges; see
Gubeladze \cite{Gub} for an edge length bound guaranteeing
normality.

\subsection{Very ample non-normal polytopes}
The first example of a very ample, but non-normal lattice
polytope was given in \cite{BG:pol}. Its vertices are the
$0$-$1$-vectors representing the triangles of the minimal
triangulation of the real projective plane. The polytope has
dimension $5$.

Very ample non-normal polytopes can easily be found by
\emph{shrinking}. One starts from a normal polytope $P$,
chooses a vertex $v$, and checks whether the polytope $Q$
spanned by $(P\cap\ZZ^d)\setminus\{v\}$ is very ample. If so,
$P$ is replaced by $Q$. If not, we test another vertex of $P$.
If no vertex of $P$ can be removed without violating very
ampleness for $Q$, one stops at $P$. The very ample polytopes
encountered in the process are checked for normality, and
surprisingly often non-normal very ample polytopes pop up, and
even smooth ones do. (This technique was originally applied to
find normal polytopes without unimodular cover; see
\cite{BG:cov}.)

By shrinking we found the following polytope: $P\subset\RR^3$
is the convex hull of
$$
\bigl((0,0)\times I_1\bigr)\cup \bigl((0,1)\times I_2\bigr)\cup
\bigl((1,1)\times I_3\bigr)\cup \bigl((1,0)\times I_4\bigr).
$$
with $ I_1=\{0,1\}$, $I_2=\{2,3\}$, $I_3=\{1,2\}$,
$I_4=\{3,4\}$ (see \cite[Exerc.\ 2.24]{BG}). This polytope has
$4$ unimodular corner cones and $4$ non-simple ones. One can
check by hand that at each vertex $v$ the vectors $w-v$, $w$ a
neighbor of $v$, form $\Hilb(M(P)_v)$. This example has
recently been generalized by Ogata \cite{Ogata} in several
ways. Very ampleness for these polytopes can be proved by
applying the following criterion to each of the non-unimodular
corner cones:

\begin{proposition}
Let $C$ be a rational cone of dimension $d$ generated by $d+1$
vectors $w_1,\dots,w_{d+1}$. For each facet $F$ of $C$ suppose
that the $w_i \in F$ together with one of the (at most two)
remaining ones generate $\ZZ^d$. Then $w_1,\dots,w_{d+1}$ are
the Hilbert basis of $C$.
\end{proposition}

\begin{proof}
The hypothesis guarantees that $w_1,\dots,w_{d+1}$ generate
$\ZZ^d$. Moreover, the monoid $C\cap\ZZ^d$ is integral over the
monoid $M$ generated by the $w_i$. Therefore it is enough to
show that $M$ is normal. The hypothesis on the generation of
$\ZZ^d$ by the $w_i\in F$ and \emph{one} additional vector
implies that the monoid algebra $K[M]$ satisfies Serre's
condition $(R_1)$ (compare \cite[Exerc. 4.16]{BG}). Serre's
condition $(S_2)$ is satisfied since an affine domain of
dimension $d$ generated by $d+1$ elements is Cohen-Macaulay.
Normality is equivalent to $(R_1)$ and $(S_2)$.
\end{proof}

\subsection{The search for very ample simple polytopes}
Finding random simple polytopes has turned out as difficult as
finding random smooth polytopes. Constructing such polytopes
from simplicial fans follows the algorithm outlined in Section
\ref{Fans}, except that one does not refine a simplicial cone
to a unimodular one. However, the property (LP) now comes into
play, and the ``right hand sides"' $b$ that yield lattice
polytopes (and not just rational ones) form a proper sublattice
$\cC$ of $\cW$. It is not hard to describe $\cC$ by congruences
that its members must satisfy. However, often the Hilbert basis
computation is arithmetically much more complicated than for
unimodular fans, and the way out described in Remark
\ref{Pract} does not work well.

Despite of the fact that simple lattice polytopes are usually
not normal, those that we have obtained from fans have all been
normal. This fact reveals the most problematic aspect of our
search: creating polytopes from random simplicial or even
smooth fans seems to produce only harmless examples since one
cannot reach the complication, arithmetically or
combinatorially, that $\cN(P)$ needs for $P$ to be non-normal.
It should be much more promising to define polytopes in terms
of their vertices.

Simplices are the only class of simple polytopes that can
easily be produced by choosing vertices at random. In dimension
$\ge 3$ they are usually non-normal, but we have not yet been
able to find a very ample such simplex. The only result known
to us that indicates that simplices are special in regard to
very ampleness is a theorem of Ogata \cite{OgaWeight}: if $P$
is a very ample simplex of dimension $d$, then the multiples
$cP$ are normal for $c\ge n/2-1$. In particular, very ample
$3$-simplices are normal.


\begin{thebibliography}{99}

\bibitem{BR} A.M. Bigatti, R. La Scala, and L. Robbiano.
    \emph{Computing toric ideals.} J. Symb. Comp. 27 (1999), 351--365.

\bibitem{ICP} W. Bruns. \emph{On the integral Caratheodory
    property.} Experimental Math. 16 (2008), 359--363.

\bibitem{TE} W. Bruns. \emph{ToricExp: Experiments in toric
    geometry and lattice polytopes}. Available at
    \textsf{http://www.home.uni-osnabrueck.de/wbruns/}.

\bibitem{BG:pol} W.~Bruns and J.~Gubeladze. \emph{Polytopal
    linear groups.} J. Algebra 218 (1999), 715--37.

\bibitem{BG:cov} W.~Bruns and J.~Gubeladze. \emph{Normality and
    covering properties of affine semigroups.} J. Reine Angew.
    Math. 510 (1999), 161--178.

\bibitem{BG}W.~Bruns and J.~Gubeladze. {\em Polytopes, rings,
    and K-theory.} Springer 2009.

\bibitem{BGT} W. Bruns, J. Gubeladze, and N. V. Trung.
    \emph{Normal polytopes, triangulations, and Koszul
    algebras.} J. Reine Angew. Math. 485 (1997),
    123--160.

\bibitem{BH} W.~Bruns and J. Herzog. \emph{Semigroup rings
    and simplicial omplexes.} J. Pure Appl. Algebra 122 (1997),
    185--208.

\bibitem{BI} W.~Bruns and B.~Ichim. \emph{Normaliz: algorithms
    for affine monoids and rational cones.} J. Algebra 324 (2010),
    1098--1113.

\bibitem{Nmz} W. Bruns, B. Ichim and C. Söger. \emph{Normaliz.
    Algorithms for rational cones and affine monoids.} Available
    from \textsf{http://www.math.uos.de/normaliz}.

\bibitem{BK} W. Bruns and R. Koch, {\em Computing the integral
    closure of an affine semigroup.} Univ. Iagell. Acta Math.
    {\bf 39}  (2001), 59--70.

\bibitem{CLS} D. Cox, J. Little and H. Schenck. \emph{Toric
    Varieties.} American Mathematical Society 2011.

\bibitem{DHK} J. A. De Loera, D. C. Haws and M. Köppe.
    \emph{Ehrhart polynomials of matroid poytopes and
    polymatroi ds.} Discrete Comput. Geom. 42 (2009), 670--702.

\bibitem{Ful} W. Fulton. \emph{Introduction to toric
    varieties.} PrincetonUniversity Press 1993.

\bibitem{Gub} J.~Gubeladze. \emph{Convex normality of rational
    polytopes with long edges.} Adv. Math. 230 (2012), 372--389.

\bibitem{OW} Ch. Haase, T. Hibi and D. Maclagan
    (organizers). \emph{Mini-Workshop: Projective Normality of Smooth Toric
    Varieties.} OWR 4 (2007), 2283--2320.

\bibitem{HS} S. Ho\c{s}ten and B. Sturmfels. \emph{GRIN: an
    implementation of Gröbner bases for integer programming.} In \emph{Integer
    programming and combinatorial optimization}, Lect.Notes
    Comput. Sci. 920, Springer 1995, pp. 267--276.

\bibitem{Mich} M. Michalek. \emph{Constructive degree bounds
    for group-based models.} Preprint arXiv:1207.0930.

\bibitem{Oda} T. Oda. \emph{Convex bodies and algebraic
    geometry     (An introduction to the theory of toric varieties)}.
    Springer 1988.

\bibitem{OgaWeight} Sh. Ogata. \emph{$k$-Normality of weighted
    projective spaces.} Kodai Math. J. 28 (2005), 519--524.

\bibitem{OgaProj} Sh. Ogata, \emph{Projective normality of
    toric 3-folds with non-big adjoint hyperplane sections.}
    Tohoku Math. J. 64 (2012), 125--140.

\bibitem{OgaFiber} Sh. Ogata. \emph{PAmple line bundles on a
    certain toric fibered 3-folds.} Preprint,
    arXiv:1104.5573v1.

\bibitem{Ogata} Sh. Ogata. \emph{Very ample but not integrally
    closed lattice polytopes.} Beitr. Algebra Geom., to appear.

\bibitem{Schw} J. Schweig. \emph{Toric ideals of lattice path
    matroids and polymatroids.} J. Pure Appl. Algebra 215 (2011),
    2660--2665.

\bibitem{St} R.P. Stanley. \emph{Combinatorics and commutative
    algebra, second ed.} Birkhäuser 1996.


\end{thebibliography}
\end{document}